\documentclass[a4paper]{article}
\usepackage{amsthm,amssymb,amsmath,enumerate,graphicx,epsf,bbold,tabularx,rotating}

\newcommand{\COLORON}{1}
\newcommand{\NOTESON}{0}
\newcommand{\Debug}{0}
\usepackage[usenames,dvips]{color} 
\usepackage{amsthm,amssymb,amsmath,enumerate,graphicx,epsf,mathbbol,stmaryrd}
\usepackage[dvips, bookmarks, colorlinks=false, breaklinks=true]{hyperref}

\newcommand{\tcon}{3-connected}

\newcommand{\sepe}{hinge}
\newcommand{\pr}{consistent}
\newcommand{\prem}{\pr\ embedding}
\newcommand{\vapf}{VAP-free}
\newcommand{\ccfg}{\ensuremath{\cc_f(\G)}} 
\newcommand{\idg}{\ensuremath{\mathbb 1}}
\newcommand{\red}{simple}

\newcommand{\auto}{colour-automorphism}

\newcommand{\dray}{double ray}

\newcommand{\comment}[1]{}
\newcommand{\COMMENT}[1]{}

\definecolor{darkgray}{rgb}{0.3,0.3,0.3}
\newcommand{\defi}[1]{{\color{darkgray}\emph{#1}}}

\newcommand{\acknowledgement}{\section*{Acknowledgement}}

\comment{
	\begin{lemma}\label{}	
\end{lemma}
\begin{proof}

\end{proof}

\begin{theorem}\label{}
\end{theorem} 
\begin{proof} 	

\end{proof}

}

\newtheorem{proposition}{Proposition}[section]

\newtheorem{theorem}[proposition]{Theorem}
\newtheorem{corollary}[proposition]{Corollary}
\newtheorem{lemma}[proposition]{Lemma}
\newtheorem{observation}[proposition]{Observation}

\newtheorem{examp}[proposition]{Example}

\newcommand{\FIG}{0}

\ifnum \NOTESON = 1 \newcommand{\note}[1]{  

	{\color{blue} \hspace*{-60pt} NOTE: \color{Turquoise}{\small  \tt \begin{minipage}[c]{1.1\textwidth}  #1 \end{minipage} \ignorespacesafterend }} 
	
	}
\else \newcommand{\note}[1]{} \fi

\newcommand{\afsubm}[1]{ \ifnum \Debug = 1 {\mymargin{#1}}
\fi} 

\ifnum \Debug = 1 
\else  \fi

\ifnum \FIG = 1 \newcommand{\fig}[1]{Figure ``{#1}''}
\else \newcommand{\fig}[1]{Figure~\ref{#1}} \fi

\ifnum \FIG = 1 \newcommand{\figs}[1]{Figures ``{#1}''}
\else \newcommand{\figs}[1]{Figures~\ref{#1}} \fi

\ifnum \Debug = 1 \usepackage[notref,notcite]{showkeys}
\fi

\ifnum \COLORON = 0 \renewcommand{\color}[1]{}
\fi

\newcommand{\showFig}[2]{
   \begin{figure}[htbp]
   \centering
   \noindent
   \epsfbox{#1.eps}
   \caption{\small #2}
   \label{#1}
   \end{figure}
}

\newcommand{\N}{\ensuremath{\mathbb N}}
\newcommand{\R}{\ensuremath{\mathbb R}}
\newcommand{\Z}{\ensuremath{\mathbb Z}}

\newcommand{\cc}{\ensuremath{\mathcal C}}

\newcommand{\cp}{\ensuremath{\mathcal P}}

\newcommand{\oo}{\ensuremath{\omega}}

\newcommand{\Gam}{\ensuremath{\Gamma}}

\newcommand{\sig}{\ensuremath{\sigma}}

\newcommand{\sm}{\backslash}
\newcommand{\restr}{\upharpoonright}

\newcommand{\isom}{\cong}

\newcommand{\pth}[2]{\ensuremath{#1}\text{--}\ensuremath{#2}~path}

\newcommand{\pths}[2]{\ensuremath{#1}\text{--}\ensuremath{#2}~paths}

\newcommand{\g}{\ensuremath{G\ }}
\newcommand{\G}{\ensuremath{G}}

\newcommand{\Lr}[1]{Lemma~\ref{#1}}
\newcommand{\Tr}[1]{Theorem~\ref{#1}}
\newcommand{\Sr}[1]{Section~\ref{#1}}

\newcommand{\Cr}[1]{Corollary~\ref{#1}}

\newcommand{\Cg}{Cayley graph}

\renewcommand{\iff}{if and only if}
\newcommand{\fe}{for every}

\newcommand{\st}{such that}

\newcommand{\ti}{there is}
\newcommand{\ta}{there are}

\newcommand{\obda}{without loss of generality}

\newcommand{\wrt}{with respect to}

\newcommand{\labtequ}[2]{ \begin{equation} \label{#1} 	\begin{minipage}[c]{0.9\textwidth}  #2 \end{minipage} \ignorespacesafterend \end{equation} }

\newcommand{\mymargin}[1]{
  \marginpar{%
    \begin{minipage}{\marginparwidth}\small%
      \begin{flushleft}%
        {\color{blue}#1}%
      \end{flushleft}%
   \end{minipage}%
  }%
}%

\newcommand{\mySection}[2]{}

\newcommand{\DB}{\cite{diestelBook05}}

\newcommand{\citeTIIc}{\cite[Theorem~11.10]{cayley3}}

\title{The planar cubic Cayley graphs of connectivity 2}
\author{Agelos Georgakopoulos\thanks{Supported by FWF grant P-19115-N18.} \medskip \\
  {Technische Universit\"at Graz}\\
  {Steyrergasse 30, 8010}\\
  {Graz, Austria}\\
}

\begin{document}
\maketitle

\begin{abstract}

We classify the planar cubic Cayley graphs of connectivity 2, providing an explicit presentation and embedding for each of them. Combined with \cite{cayley3} this yields a complete description of all planar cubic Cayley graphs.

\end{abstract}

\section{Introduction}
This paper is part of a series of related papers on planar Cayley graphs \cite{vapf,cayley3,am,agmh}. A major aim of this project is to provide a complete description of the \defi{cubic} planar Cayley graphs, i.e.\ those in which every vertex is adjacent with precisely three other vertices. This analysis provides surprising new examples, contradicting past conjectures, but also a good insight into planar Cayley graphs in general; see \cite{cayley3}. 

The aim of the current paper is a classification of the planar cubic \Cg s of connectivity 2, yielding an explicit presentation and embedding for each of those graphs. The \defi{connectivity} of a graph \g is the smallest cardinality of a set of vertices separating \G. \Cg s of connectivity 1 are easy to describe. It is very common \cite{DrSeSeCon,dunPla} to consider graphs of connectivity 2 separately from graphs of higher connectivity when studying planar \Cg s, and this is so for a good reason: by a classical theorem of Whitney \cite[Theorem 11]{whitney_congruent_1932}, planar graphs of connectivity at least 3 have a unique, up to homeomorphism, embedding in the sphere $S^2$, and can thus be analysed taking advantage of this fact.

Our main result is
\begin{theorem} \label{Tmain}
Let \g be a planar cubic \Cg\ of connectivity 2. Then precisely one of the following is the case:
\begin{enumerate}\addtolength{\itemsep}{-0.5\baselineskip}
\item \label{i} $G \isom Cay \left<a,b\mid b^2, (ab)^n\right>$, $n\geq 2$; 
\item \label{ii} $G \isom Cay \left<a,b\mid b^2, (aba^{-1}b^{-1})^n\right>$, $n\geq 1$;
\item \label{iii} $G\isom Cay \left<a,b \mid b^2, a^4, (a^2b)^n \right>, n\geq 2$;

\item \label{iv} $G \isom Cay \left< b,c,d \mid b^2, c^2, d^2, (bc)^2, (bcd)^m\right>$, $m\geq 2$;
\item \label{v} $G \isom Cay \left<b,c,d \mid b^2, c^2, d^2, (bc)^{2n}, (cbcd)^m\right>$, $n,m\geq 2$;
\item \label{vi} $G \isom Cay \left<b,c,d\mid b^2, c^2,d^2, (bc)^n, (bd)^m\right>$, $n,m\geq 2$;
\item \label{vii} $G \isom Cay \left<b,c,d\mid b^2, c^2,d^2, (b(cb)^n d)^m\right>, n,m\geq 2$;
\item \label{viii} $G \isom Cay \left<b,c,d\mid b^2, c^2,d^2, (bcbd)^m \right>$, $m\geq 1$;
\item \label{ix} $G \isom Cay\left<b,c,d\mid b^2, c^2, d^2, (bc)^n, cd\right>$, $n\geq 1$ (degenerate cases with redundant generators and \g finite). 
\end{enumerate}
Conversely, each of the above presentations, with parameters chosen in the specified domains, yields a planar cubic \Cg\ of connectivity 2.
\end{theorem}

Except for the above presentations we also construct embeddings of these  graphs \st\ 
the corresponding group action on the graph carries facial walks to facial walks (\Cr{cprem}); see also \Sr{secDem}. 
Tables~\ref{table1} and \ref{table2} summarise some information about these embeddings and some basic properties of the graphs, yielding a  more structured presentation of the various possibilities of \Tr{Tmain}. The interested reader will also find further information, not included in these tables, throughout  the course of the proofs. Moreover, in the last section of the paper we point out some interesting corollaries, which are extended in \cite{cayley3} to all cubic planar \Cg s.

This paper contributes to the complete classification of the cubic planar \Cg s of \cite{cayley3} not only by settling the special case of graphs of connectivity two, but also by providing building blocks for the construction of some of the \tcon\ ones. For example, the graphs of type \ref{vi} of \Tr{Tmain} and their embeddings that we construct here are used in \cite{cayley3} to produce \tcon\ \Cg s that have no finite face boundaries, thus contradicting a conjecture of Bonnington and Watkins \cite{BoWaPla}.


\begin{sidewaystable}[h!b!p!]

\setlength{\extrarowheight}{10pt} 
\begin{tabular}{|p{.25\linewidth}|p{.25\linewidth}|p{.25\linewidth}|p{.25\linewidth}|}
\hline
\multicolumn{4}{|c|}{$G = Cay\left<a,b \mid b^2,\ldots\right>$, $G$ is planar, and $\kappa(G)=2$}  \\
\hline

\multicolumn{2}{|p{.5\linewidth}|}{$a^n=1$.}%
	
	& \multicolumn{2}{m{.5\linewidth}|}{$a^n\neq1$. Then $G$ has a \prem\ $\sig$ in which $a$ preserves spin.} 	\\
\hline

\multicolumn{2}{|p{.5\linewidth}|}{ $\Gamma\isom  \left<a,b \mid b^2, a^4, (a^2b)^n \right>, n\geq 2$. Thus \g has a \prem\ in which $a$ reverses spin and $b$ preserves spin.}
	& $\Gamma\isom \left<a,b\mid b^2, (ab)^n\right>$, $n\geq 2$ and $b$ preserves spin in \sig.	&	$\Gamma \isom \left<a,b\mid b^2, (aba^{-1}b^{-1})^n\right>$, $n\geq 1$ and $b$ reverses spin in \sig. \\
\hline
\end{tabular}
\caption{Classification of the cubic planar Cayley graphs with 2 generators and connectivity 2 (types \ref{i} to \ref{iii} of \Tr{Tmain}).}
\label{table1}

\

\


\setlength{\extrarowheight}{10pt} 
\begin{tabular}{|p{.2\linewidth}|p{.2\linewidth}|p{.2\linewidth}|p{.2\linewidth}|p{.2\linewidth}|}
\hline
\multicolumn{5}{|c|}{$G = Cay\left<b,c,d \mid b^2,c^2,d^2,\ldots\right>$, $G$ is planar, and $\kappa(G)=2$}  \\
\hline

\multicolumn{3}{|p{.6\linewidth}|}{$(bc)^n=1$ (i.e.\ $G$ has a 2-coloured cycle).}%
	
	& \multicolumn{2}{|p{.4\linewidth}|}{$G$ has no 2-coloured cycle.} \\
\hline

$G$ has no \sepe, $\Gam \isom \left< b^2, c^2, d^2 \mid (bc)^2, (bcd)^m\right>$, $m\geq 2$, and \g has a \prem\  in which all edges preserve spin.		
	&	$G$ has no \sepe, $\Gam\isom \left<b^2, c^2, d^2 \mid (bc)^{2n}, (cbcd)^m\right>$, $n,m\geq 2$, and \g has a \prem\  in which $c$ preserves spin and $b,d$ reverse spin.
	&	$G$ has a \sepe, $\Gam\isom \left<b^2, c^2, d^2\mid (bc)^n, (bd)^m\right>$, $n,m\geq 2$, and \g has a \prem\ in which all edges reverse spin.
	&	$G$ has no \sepe, $\Gam\isom \left<b^2, c^2, d^2\mid (b(cb)^nd)^m\right>$, $n,m\geq 2$,  and \g has an embedding in which all edges preserve spin.
	& 	$G$ has a \sepe\ $\Gam\isom \left<b^2, c^2, d^2\mid (bcbd)^n \right>$, $n\geq 1$, and \g has an embedding in which only $b$ preserves spin.\\

\hline
	
\end{tabular}
\caption{Classification of the infinite cubic planar Cayley graphs with 3 generators and connectivity 2 (types \ref{iv} to \ref{viii} of \Tr{Tmain}).}
\label{table2}
\end{sidewaystable}

\medskip

In many cases we obtain embeddings of our graphs in the sphere using only the fact that their connectivity is 2, without assuming a priori that the graphs are planar. Thus we obtain the following result, which describes the non-planar cubic \Cg s of connectivity 2.

\begin{corollary} \label{cornp}
Let \g be a non-planar cubic \Cg\ of connectivity 2. Then one of the following is the case:
\begin{enumerate}
\item $G \isom Cay\left<a,b\mid b^2, a^n, \ldots \right>$, $n>2$; or
\item $G\isom Cay\left<b,c,d\mid b^2, c^2, d^2, (bc)^k, \ldots \right>$, $k\geq 3$, \g has no \sepe, and every $bc$ cycle contains a pair $x,y$ of vertices separating the graph \st\ $x^{-1}y$ is an involution.
\end{enumerate}
\end{corollary}


\section{Definitions}

\subsection{\Cg s and group presentations}
We will follow the terminology of \cite{diestelBook05} for graph-theoretical terms and that of \cite{bogop} for group-theoretical ones. Let us recall the definitions most relevant for this paper.

Let \Gam\ be group and let $S$ be a symmetric generating set of \Gam. The Cayley graph $Cay(\Gam,S)$ of $\Gam$ \wrt\ $S$ is a coloured directed graph $\g= (V,E)$ constructed as follows. The vertex set of \g is \Gam, and the set of colours we will use is $S$.  For every $g\in \Gam, s\in S$ join $g$ to $gs$ by an edge coloured $s$ directed from $g$ to $gs$. Note that $\Gam$ acts on \g by multiplication on the left; more precisely, \fe\ $g\in \Gam$ the mapping from $V(G)$ to $V(G)$ defined by $x \mapsto gx$ is a \defi{\auto} of \G, that is, an automorphism of \g that preserves the colours and directions of the edges. In fact, \Gam\ is precisely the group of \auto s of \G.

If $s\in S$ is an \defi{involution}, i.e.\ $s^2=1$, then every vertex of \g is incident with a pair of parallel edges coloured $s$ (one in each direction). However, for simplicity we will draw only one, undirected, edge in this case.

We call a graph \defi{cubic}, if each of its vertices is adjacent with precisely three other vertices. 


Given a group presentation $\left< S \mid \mathcal R \right>$ we will use the notation $Cay \left< S \mid \mathcal R \right>$ for the \Cg\ of this group \wrt\ $S$.

If $R\in \mathcal R$ is any relator in such a presentation and $g$ is a vertex of $G=Cay \left< S \mid \mathcal R \right>$, then starting from $g$ and following the edges corresponding to the letters in $R$ in order we obtain a closed walk $W$ in $G$. We then say that $W$ is \defi{induced} by $R$; note that for a given $R$ \ta\ several walks in \g induced by $R$, one for each starting vertex $g\in V(G)$. If $R$ induces a cycle then we say that $R$ is \defi{\red}; note that this does not depend on the choice of the starting vertex $g$. A presentation $\left<a,b,\ldots \mid R_1,R_2,\ldots\right>$ of a group \Gam\ is called \defi{\red}, if $R_i$ is \red\ \fe\ $i$. In other words, if for every $i$ no proper subword of any $R_i$ is a relation in $\Gamma$. 

Define the (finitary) \defi{cycle space} \ccfg\ of a graph $G=(V,E)$ to be the vector space over $\Z_2$ consisting of those subsets of $E$ such that can be written as a sum (modulo 2) of a finite set of {circuits}, where a set of edges $D\subseteq E$ is called a \defi{circuit} if it is the edge set of a cycle of \G. Thus \ccfg\ is isomorphic to the first simplicial homology group of \g over $\Z_2$. The \defi{circuit} of a closed walk $W$ is the set of edges traversed by $W$ an even number of times. Note that the direction of the edges is ignored when defining circuits and \ccfg. The cycle space will be a useful tool in our study of \Cg s because of the following well-known fact which is easy to prove.
\begin{lemma} \label{relcc}
Let $G= Cay \left< S \mid R \right>$ be a \Cg\ of the group \Gam. Then the set of circuits of walks in \g induced by relators in $R$ generates $\cc_f(G)$. 

Conversely, if $R'$ is a set of words, with letters in a set $S\subseteq \Gam$ generating \Gam, such that the set of circuits of cycles of $Cay(G,S)$ induced by $R'$ generates \ccfg, then $\left< S \mid R' \right>$ is a presentation of \Gam.
\end{lemma}

\subsection{Graph-theoretical concepts}
Let $G=(V,E)$ be a connected graph fixed throughout this section. Two paths in \g are \defi{independent}, if they do not meet  at any vertex except perhaps at common endpoints. 
If $P$ is a path or cycle we will use $|P|$ to denote the number of vertices in $P$ and  $||P||$ to denote the number of edges of $P$. Let $xPy$ denote the subpath of $P$ between its vertices $x$ and $y$.

A \defi{\dray} is a two-way infinite path in a graph.

A \defi{\sepe} of \g is an edge $e=xy$ \st\ the removal of the pair of vertices $x,y$  disconnects \G. A \sepe\ should not be confused with a \defi{bridge}, which is an edge whose removal separates \g  although its endvertices are not removed.

The set of neighbours of a vertex $x$ is denoted by \defi{$N(x)$}.

\G\ is called \defi{$k$-connected} if $G - X$ is connected for every set $X\subseteq V$ with $|X | < k$. Note that if \g is $k$-connected then it is also $(k-1)$-connected. The \defi{connectivity $\kappa(G)$} of \g is the greatest integer $k$ such that \G\ is $k$-connected.

\subsection{Embeddings into the plane} \label{secDem}	

An \defi{embedding} of a graph \g will always mean a topological embedding of the corresponding 1-complex in the euclidean plane $\R^2$; in simpler words, an embedding is a drawing in the plane with no two edges crossing.

A \defi{face} of an embedding $\sig: G \to \R^2$ is a component of $\R^2 \sm \sig(G)$. The \defi{boundary} of a face $F$ is the set of vertices and edges of \g that are mapped by \sig\ to the closure of $F$. The \defi{size} of $F$ is the number of edges in its boundary. Note that if $F$ has finite size then its boundary is a cycle of \G.

A walk in \g is called \defi{facial} \wrt\ \sig\ if it is contained in the boundary of some face of \sig. 

An embedding of a \Cg\ is called \defi{\pr} if, intuitively, it embeds every vertex in a similar way in the sense that the group action carries faces to faces. Let us make this more precise.
Given an embedding \sig\ of a \Cg\ $G$ with generating set $S$, we consider \fe\ vertex $x$ of \g the embedding of the edges incident with $x$, and define the \defi{spin} of $x$ to be the cyclic order of the set $L:=\{xy^{-1} \mid y\in N(x)\}$ in which $xy_1^{-1}$ is a successor of $xy_2^{-1}$ whenever the edge $xy_2$ comes immediately after the edge $xy_1$ as we move clockwise around $x$. Note that the set $L$ is the same \fe\ vertex of $G$, and depends only on $S$ and on our convention on whether to draw one or two edges per vertex for involutions. This allows us to compare spins of different vertices. Note that if \g is cubic, which means that $|L|=3$, then \ta\ only two possible cyclic orders on $L$, and thus only two possible spins. Call an edge of \g \defi{spin-preserving} if its two endvertices have the same spin in \sig, and call it \defi{spin-reversing} otherwise. Call a colour in $S$ \defi{\pr} if all edges bearing that colour are  spin-preserving or all edges bearing that colour are spin-reversing in \sig. Finally, call the embedding \sig\ \defi{\pr} if every colour is \pr\ in \sig\ (this definition is natural only if \g is cubic; to extend it to the general case, demand that every two vertices have either the same spin, or the spin of the one is obtained by reversing the spin of the other). 

It is straightforward to check that \sig\ is \pr\ \iff\ every \auto\ of \g maps every facial walk to a facial walk. 

It follows from Whitney's theorem mentioned in the introduction that if \g is \tcon\ then its essentially unique embedding must be \pr. \Cg s of connectivity 2 do not always admit a \prem\ \cite{DrSeSeCon}. However, in the cubic case they do (see \Cr{cprem}), and a significant part of this paper is concerned with constructing these embeddings.

An embedding is \defi{Vertex-Accumulation-Point-free}, or \defi{\vapf} for short, if the images of the vertices have no accumulation point in $\R^2$. The following fact will allow us to easily deduce that certain \Cg s are planar looking only at the corresponding presentations.

\begin{corollary}[{\cite[Section 4]{vapf}}] \label{Macay}
Let  $\left< S \mid R \right>$ be a \red\ presentation and let $G= Cay \left< S \mid R \right>$ be the corresponding \Cg. If no edge of \g appears in more than two  circuits  induced by relators in $R$, then $G$ is planar and has a \vapf\ embedding the facial cycles of which are precisely the cycles of \g induced by relators in $R$.
\end{corollary}

\comment{
	\section{MacLane's planarity criterion for \Cg s}

	A nice application of $\cc_f(G)$ is MacLane's classical theorem (see \cite{maclane37} or \DB), characterising the finite planar graphs as those graphs \g for which $\cc_f(G)$ has a \defi{2-basis}, that is, a generating set $B$ \st\ no edge of \g appears in more than two elements of $B$. The following theorem of Thomassen extends MacLane's planarity criterion to infinite graphs. 

\begin{theorem}[{\cite[Theorem 7.4.]{thoPla}}] \label{thobas}
Let $G$ be an infinite 2-connected graph. Then \g has a 2-basis if and only if $G$ is planar and has a \vapf\ embedding.
Moreover, if $B$ is any 2-basis of $G$, then $G$ has a \vapf\ embedding \sig\
such that $B$ is the set of facial cycles of \sig. 
\end{theorem}

This, combined with \Lr{relcc}, yields the following valuable tool, which will allow us to easily deduce that certain \Cg s are planar looking only at the corresponding presentations.

	\begin{corollary} \label{Macay}
	Let  $\left< S \mid R \right>$ be a \red\ presentation and let $G= Cay \left< S \mid R \right>$ be the corresponding \Cg. If no edge of \g appears in more than two  circuits  induced by relators in $R$, then $G$ is planar and has a \vapf\ embedding the facial cycles of which are precisely the cycles of \g induced by relators in $R$.
	\end{corollary}
} 


\section{Cayley graphs with 2 generators} \label{sec2}

In the rest of the paper \g will always denote a \Cg\ and $\Gam(G)$, or just \Gam, will denote the corresponding group. In this section we consider cubic \Cg s on two generators $a,b$ where $b$ is an involution.

The following lemma will play an important role.

\begin{lemma}\label{Lk}	
Let $G= Cay\left<a,b\mid b^2, \ldots \right>$ be a Cayley graph with $\kappa(G)=2$, and let $P$ be a shortest path whose endvertices $x,y$ separate \G. Then $x^{-1}y$ is an involution, and it equals $b$ or $a^i$ for some $i\in\Z$. Moreover, $x$ sends three independent paths to $y$.
\end{lemma}
\begin{proof}
We begin by proving that 
\labtequ{foure}{The component $C$ of $G - \{x,y\}$ meeting $P$ sends four edges to $\{x,y\}$.}
Indeed, suppose that $x$ sends only one edge $xx'$ to $C$. Then $P$ must contain the edge $xx'$ and, easily, $\{x',y\}$ also separates \G. This contradicts the choice of $P$, since the path $P - x$ is shorter and its endvertices also separate \G. Thus $x$ must send two edges to $C$, and by the same argument $y$ must send another two edges to $C$, which proves \eqref{foure}.

Our next aim is to prove that 
\labtequ{thrpaths}{There are 3 independent \pths{x}{y} in \G.}
To see this, note that no vertex of $C$ separates $N(x)$ from $N(y)$ in $C$. Indeed, if $v$ is such a vertex, then $\{x,v\}$ also separates \G. But $v$ must lie on $P$ since $P$ connects $N(x)$ to $N(y)$ in $C$. This contradicts the choice of $P$, since $xPv$ is a shorter candidate. By Menger's theorem \cite{diestelBook05}, this implies that there are two independent \pths{x}{y}\ $P_1,P_2$ through $C$. Note that by \eqref{foure} $\{x,y\}$ sends only two edges to $G - C$. If these edges are incident with two distinct components of $G - \{x,y\}$, then any of these edges is a bridge of \G, contradicting the fact that \g is 2-connected. Thus these edges are incident with the same component $C'\neq C$ of $G - \{x,y\}$, and there is an \pth{x}{y}\ $P_3$ in $G - C$. Combined with the paths $P_1,P_2$ this proves \eqref{thrpaths}.

By \eqref{foure} each of $x,y$ has a unique neighbour in $G - C$; denote these neighbours by $x',y'$ respectively. Next, we claim that 
\labtequ{Pmono}{$P$ is monochromatic.}
Indeed, if $P$ contains both colours $a,b$, then it contains an edge $e$ that has the same colour as the edge $xx'$. Thus we can translate $e$ to $xx'$ by a \auto\ $w$ of \G. Let $P':= w(P)$, and distinguish the following two cases: either the two endvertices of $P'$ are in distinct components of $G - \{x,y\}$, or not. In the first case, we obtain a contradiction to \eqref{thrpaths} since $\{x,y\}$ separates the endvertices of $P'$. In the second case, we distinguish two subcases: either $P'$ contains the edge $yy'$, or one of the endvertices of $P'$ is $x$ and the other lies in the component $C'$ of $G - \{x,y\}$; no other alternative exists because $P'$ contains $xx'$ and so if $x$ is not an endpoint then $P'$ meets both $C$ and $C'$, and the edges $xx',yy'$ separate $C$ from $C'$. Now in the first of these subcases, we obtain a subpath $P''$ of $P'$ with endvertices $x,y$ that meets $C'$. This contradicts the minimality of $P$ if $P''\neq P'$ or, proposition \eqref{foure} if $P''=P'$. In the second subcase, as $xx'$ is the only $x$-$C'$ edge incident with $x$ and the endvertices $x,v$ of $P'$ separate \G, it is easy to see that the vertices $x',v$ also separate \G. Since both these vertices lie on $P'=w(P)$, we obtain a contradiction to the choice of $P$ since $x'P'v$ is a shorter candidate. Thus we obtained a contradiction in all of these cases, and so \eqref{Pmono} must hold. 
\medskip

Now if the colour of $P$ is $b$, then as $b$ is an involution $P$ consists of a single $b$ edge and we obtain $z=x^{-1}y=b$.

Suppose that every edge of $P$ is labelled $a$ instead. If any of the edges $xx',yy'$ is also labelled $a$ then we can repeat the above argument to obtain a contradiction. Thus both $xx',yy'$ are labelled $b$ in this case. This means that the coset $D$ of the subgroup spanned by $a$ containing $x$, and thus also $y$ since $P$ is labelled $a$, is contained in $\overline{C}:= C \cup \{x,y\}$. Now consider the element $z:= x^{-1}y$ of $\Gamma(G)$; note that $yz$ lies on $D\subseteq C$. Recall that \ti\ an \pth{x}{y}\ $P_3$ in $G - C$ (see \eqref{thrpaths}), and note that any such $P_3$ must begin with the edge $xx'$ and finish with $y'y$. Choose $P_3$ so as to minimize its length. Now consider its translate $P_3z$ starting at $y$ instead of $x$. By the choice of $z$ this path begins with the edge $yy'$ and terminates in $\overline{C}$. Since $xx',yy'$ are the only edges connecting $C'$ to the rest of the graph, and since $y'$ lies in $C'$, the path $P_3z$ must also contain $xx'$. Then $P_3z$ has to  terminate at $x$, because otherwise the minimality of $|P_3|$ is contradicted by the subpath of $P_3z$ from $y$ to $x$. This means that $z=x^{-1}y$ is an involution, as it exchanges $x$ with $y$. By the choice of $z$ we have $z=a^{\pm ||P||}$ in this case as desired. 

\end{proof}

\begin{theorem}\label{Tki}	
Let $G= Cay\left<a,b\mid b^2, \ldots \right>$ be a Cayley graph with $\kappa(G)=2$ in which $a$ has infinite order. Then either $\Gamma(G)\isom \left<a,b\mid b^2, (ab)^n\right>$ with $n\geq 2$,  or $\Gamma(G)\isom \left<a,b\mid b^2, (aba^{-1}b^{-1})^n\right>$ with $n\geq 1$. Thus $G$ is planar, and has a \prem\ in which $a$ preserves spin.

Conversely, each of the above presentations, with parameters chosen in the specified domains, yields a planar cubic \Cg\ of connectivity 2 in which $a$ has infinite order.
\end{theorem}
\begin{proof}
Let $P$ be a shortest path whose endvertices $x,y$ separate \G. By \Lr{Lk} $x^{-1}z$ is an involution, and it equals $b$ or $a^i$ for some $i\in\Z$. But since $a$ has infinite order we have $a^{2i}\neq 1$, and so only the former can be the case. This proves that
\labtequ{bsep}{Every $b$ edge is a \sepe.}
We now have enough information about the structure of \g to enable as to work out a presentation of $\Gamma(G)$. We start by proving that
\labtequ{noaa}{There is a simple presentation $\left<a,b \mid R_1,R_2,\ldots\right>$ of $\Gamma(G)$ such that $a^2 \not \preceq R_i$ holds for every $i$,}
where we write $w \preceq R$ if any relator obtained $R$ from $R$ by rotation and inversion contains the word $w$, and write $w \not\preceq R$ otherwise.

To begin with, let  $\left<a,b \mid R_1,R_2,\ldots\right>$ be any simple presentation of $\Gamma(G)$. Suppose that $R_i$ contains the subword $a^2$. Consider a closed walk $W$ in \g induced by $R_i$. Thus $W$ traverses two consecutive edges $zx,xv$ labelled $a$. Now let $e=xy$ be the $b$ edge incident with $x$. By \eqref{bsep} $e$ is a \sepe. If the edges $zx,xv$ are incident with the same component of $G - \{x,y\}$, then it is easy to see that one of the edges incident with $y$ is a bridge, contradicting the 2-connectedness of \G\ (\fig{chair}). Thus $zx,xv$ are incident with distinct components of $G - \{x,y\}$, which implies that $W$, being a closed walk, must visit $y$.

\showFig{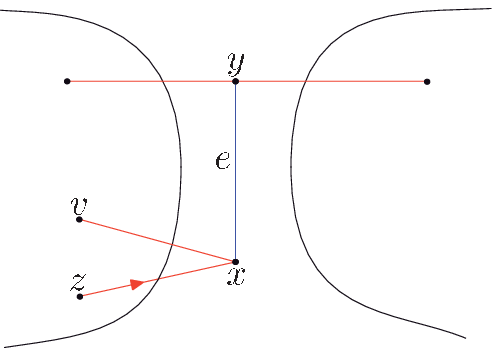}{The situation in the proof of \eqref{noaa}.}

Now as $W$ contains both $x,y$, we can use the edge $e=xy$ to `shortcut' $W$ into two walks $W_1,W_2$, both starting and ending at $y$, such that the concatenation of the corresponding relations $R^1,R^2$ yields a word equivalent to $R_i$ (\fig{W}). Note that $R^1,R^2$ are simple since $R_i$ was. Thus we could replace $R_i$ by two relations $R^1,R^2$ to obtain a simple presentation of $\Gamma(G)$ in which the word $a^2$ appears less often. We can repeat this procedure as often as needed to replace $R_i$ by a finite family of relations $R^1,\ldots R^k$ in which the words $a^2$ and $a^{-2}$ do not appear at all. Doing so for each $i$ we obtain a presentation that establishes \eqref{noaa}. Note that we did not have to assume that $\Gamma(G)$ is finitely presented; this will follow soon.

\showFig{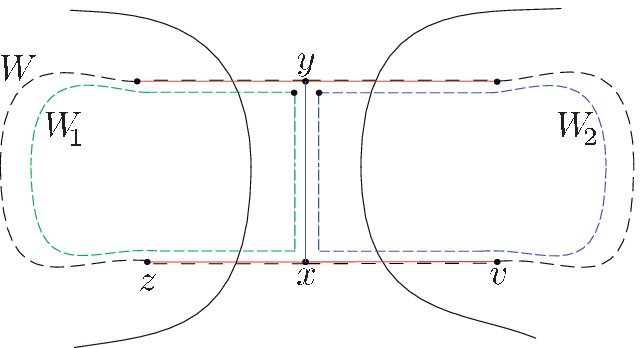}{The situation in the proof of \eqref{noaa}.}

So let $\left<a,b \mid R_1,R_2,\ldots\right>$ be a presentation of $\Gamma(G)$ as supplied by \eqref{noaa}.
Next, we claim that 
\labtequ{noaba}{There is no pair $i,j\in \N$ \st\  $aba \preceq R_i$ and $aba^{-1}\preceq R_j$.}
Indeed, suppose that $ab \preceq R_i$, and consider a cycle $K$ in \g induced by to $R_i$ such that both edges in a subpath $Q$ of $K$ induced by $ab$ are incident with the identity $\idg\in \Gamma(G)$. Applying \eqref{bsep} again to the $b$ edge $e$ incident with \idg, we obtain a separation of \g as in \fig{W}, with $x$ playing the role of \idg. Now since $K$ contains the subpath $Q$ containing $e$, and $K$ is not allowed to visit any vertex twice, the edge $f$ following $Q$ in $K$ is determined: it must be the other $a$ edge incident with the component meeting $Q$. This means that once we see the subword $ab$ in some relator $R_i$ the following letter (either $a$ or $a^{-1}$) is determined. This proves \eqref{noaba}.

Combining \eqref{noaa} with \eqref{noaba} it follows that once we see the letter $a$ or $a^{-1}$ in some $R_i$ the rest of the word is uniquely determined. In other words, there is, up to rotation of the letters and inversion, only one possible \red\ relation in $\Gamma(G)$ except for $b^2$: this relation is either $(ab)^n$ for some $n\geq 2$ or $(aba^{-1}b^{-1})^m$ for some $m\geq 1$.


Let us postpone the rest of the forward implication to prove the converse implication first.  Given one of the above group presentations, it is easy to explicitly construct a plane \Cg\ having this presentation: start with an infinite 3-regular tree $T$, and replace each vertex of $T$ by a cycle of length $2n$ or $4m$, with its edges coloured alternately $a$ and $b$. Then glue, \fe\ edge $uv$ of $T$, the two cycles replacing $u$ and $v$ along a $b$ edge. We leave the details to the reader. Note that $a$ preserves spin in any such embedding. Moreover, if $\Gamma(G)\isom \left<a,b\mid b^2, (ab)^n\right>$ then every $b$ edge preserves spin, while if $\Gamma(G)\isom \left<a,b\mid b^2, (aba^{-1}b^{-1})^m\right>$ then every $b$ edge reverses spin. It is also clear that each of the above presentations corresponds to a group in which $a$ has infinite order.

Combining what we have proved so far we can now finish off the forward implication: since \g must have one of the above presentations, it also has a \prem\ as constructed above.

\end{proof}

Next, we consider the case when the order of $a$ is finite.

\begin{theorem}\label{Tkiv}	
Let $G= Cay\left<a,b\mid b^2, a^n, \ldots \right>$ be a planar Cayley graph with $\kappa(G)=2$. Then $G \isom Cay \left<a,b \mid b^2, a^4, (a^2b)^n \right>, n\geq 2$. Thus \g has a \prem\ in which $a$ reverses spin and $b$ preserves spin. 

Conversely, \fe\ $n\geq 2$ the above presentation yields a planar cubic Cayley graph of connectivity 2 in which $a$ has order 4.
\end{theorem}
\begin{proof}
Since \g is planar it has some embedding \sig. Let $C$ be an $a$ labelled cycle induced by $a^n$. We claim that either every edge of $C$ preserves spin in \sig\ or every edge of $C$ reverses spin in \sig. For suppose not. Then we can find consecutive edges $xy,yv$ on $C$ \st\ the $b$ edges $xz,yw$ incident with $x$ and $y$ respectively lie in the same side of $C$ and the $b$ edge $vu$ incident with $v$ lies in the other side (\fig{figk4}).

\showFig{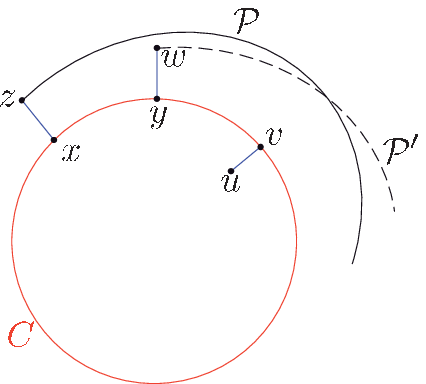}{A contradictory situation in the proof of \Tr{Tkiv}.}

We claim that $\{x,y\}$ do not separate $z$ from $w$. For if $\{x,y\}$ separates \G, then it can only leave two components behind as \g cannot contain a bridge, and each of these components sends two edges to $\{x,y\}$. Note however that both $a$-labelled edges incident with $\{x,y\}$ are incident with a common component containing $C - \{x,y\}$, and so the other two edges $xz,yw$ must go to a common component as desired.

Using this we can now prove the stronger assertion that even $C$ does not separate $z$ from $w$. Indeed, let $R$ be a \pth{z}{w}\ that does not meet $\{x,y\}$, which exists by the above claim. If $R$ meets $C$ then let $Q$ be the subpath of $R$ from $z$ to the first encounter with $C$. We can rotate $C$ by one step to map $y$ to $x$ ---by multiplying \g with an element of $\Gamma(G)$--- to obtain a translate $Q'$ of $Q$ which, by an easy topological argument, intersects $Q$. Now combining subpaths of $Q$ and $Q'$ up to such an intersection we can obtain a \pth{w}{z}\ that does not meet $C$. 

So let $S$ be a \pth{z}{w}\ that does not meet $C$. Again, we can rotate $C$ by one step to map $y$ to $v$. Now this rotation translates $S$ to a \pth{w}{u}\ which does not meet $C$. This contradicts our assumption about the embedding \sig, and proves our claim that either every edge of $C$ preserves spin or every edge reverses spin. Translating $C$ to other $a$ cycles of \g and using a similar argument we can prove that 
\labtequ{aspin}{either every $a$ edge  preserves spin or every $a$ edge reverses spin in \sig.}
%



Let $P$ be a shortest path whose endvertices $x,y$ separate \G. \Lr{Lk} implies that either $y=xb$ or $y=xa^i$ and $a^{2i}=1$. If the former case it is easy to obtain a contradiction: as the edge $xy$ cannot be a bridge, there is a path connecting the two $a$ cycles that contain $x$ and $y$; but then $\{x,y\}$ cannot separate \g since all their neighbours lie in one component. Thus $x,y$ are opposite vertices of some $a$-labelled cycle $C$. 

Since $a$ is not an involution, for otherwise we are in one of the degenerate cases of type \ref{ix} of \Tr{Tmain}, there is another pair $\{x',y'\}\neq \{x,y\}$ of opposite vertices on $C$ which also form a separator. Now as $x'$ sends three independent paths to $y'$ by   \Lr{Lk}, $\{x,y\}$ does not separate $x'$ from $y'$. Thus, all four $a$ edges incident with $x$ or $y$ are incident with the same component $K$ of $G -  \{x,y\}$, which component contains $C - \{x,y\}$. As $x$ sends three independent paths to $y$, and only two of them can meet $K$, there is an \pth{x}{y}\ $L$ that does not meet $C$ except at its endvertices. By the same argument, \ti\ also an \pth{x'}{y'}\ $L'$ that does not meet $C$ except at its endvertices. Now if $L'$ intersects $L$ we easily get a contradiction to the fact $L$ is disconnected from $K$ by $\{x,y\}$ (\fig{LLp}). Thus $L'$ cannot  intersect $L$, and an easy topological argument implies that  $L'$ and $L$ lie in different sides of $C$.

\showFig{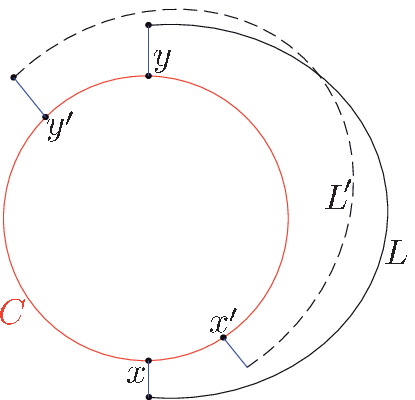}{A contradictory situation in the proof of \Tr{Tkiv}.}

Combined with \eqref{aspin} this implies that every $a$ edge reverses spin in \sig, since if $a$ edges preserve spin then all $b$ edges incident with $C$ must lie in one side of $C$. Moreover, it implies that 
\labtequ{ceqf}{$|C| = 4$.}
Indeed, otherwise there are at least three pairs of opposite vertices on $C$, and so at least two of the corresponding paths $L$ would have to share the same side of $C$.  

Let $M$ be the set of $a$-labelled cycles visited by $P$; note that $C\in M$. Easily, \fe\ cycle $D$ in $M$, only one of the sides of $D$ is met by $P$: otherwise we could shortcut $P$ using a subarc of $D$, but $P$ is chosen to have minimum length. Now note that for every vertex $v$ of \g \ti\ a $D\in M$ \st\ $v$ lies either on $D$ or in the side of $D$ not met by $P$ because, by \eqref{ceqf} and the fact that  every $a$ edge reverses spin, every edge incident with a vertex in $P \cup \bigcup M$ lies in one of those sides. Applying the same argument to the translate $P'$ of $P$ joining $xa$ to $ya$, and repeating for every $a$-labelled cycle, it is easy to see that \g has the structure of \fig{target}, and  the presentation $G = Cay \left<a,b \mid b^2, a^4, (a^2b)^m \right>$, where $m:=|M|\geq 2$.


\epsfxsize=0.35\hsize
\showFig{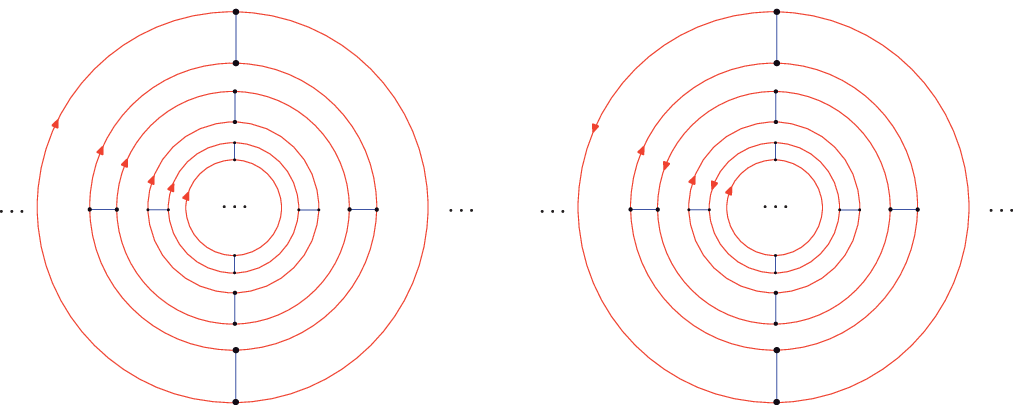}{The graph of \Tr{Tkiv} for $m=4$.}

The embedding claimed in the assertion can be seen in  \fig{target}. 
\end{proof}

\section{Cayley graphs with 3 generators} \label{sec3}

In this section we consider the case when \g is defined by three generators $b,c,d$, all of which are of course involutions since \g is cubic. We will distinguish two cases according to whether \g has a \sepe, i.e.\ an edge $e=xy$ \st\ the removal of the pair of vertices $x,y$  disconnects \G.

\subsection{Graphs with \sepe s}

\begin{theorem}\label{Tkii}	
Let $G= Cay\left<b,c,d\mid b^2, c^2,d^2, \ldots \right>$ be a Cayley graph with $\kappa(G)=2$ having a \sepe. Then either $G\isom Cay\left<b,c,d\mid b^2, c^2,d^2, (bc)^n, (bd)^m\right>$ for $m,n\geq 2$ or  $G\isom Cay\left<b,c,d\mid b^2, c^2,d^2, (bcbd)^n \right>$ for $n\geq 1$, but not both. In both cases \g is planar. In the first case it has an embedding in which  every edge reverses spin, while in the second  it has an embedding in which precisely one of the colours (that of the \sepe s) preserves spin.	

Conversely, each of the above presentations, with parameters chosen in the specified domains, yields a planar cubic \Cg\ of connectivity 2 with \sepe s.
\end{theorem}
\begin{proof}
Suppose \obda\ that every $b$ edge is a \sepe, and consider a $b$ edge $e=xy$. We distinguish two cases: either the $c$ edges incident with $x,y$ are incident with the same component of $G - \{x,y\}$ or not; see \fig{bcut}.

\showFig{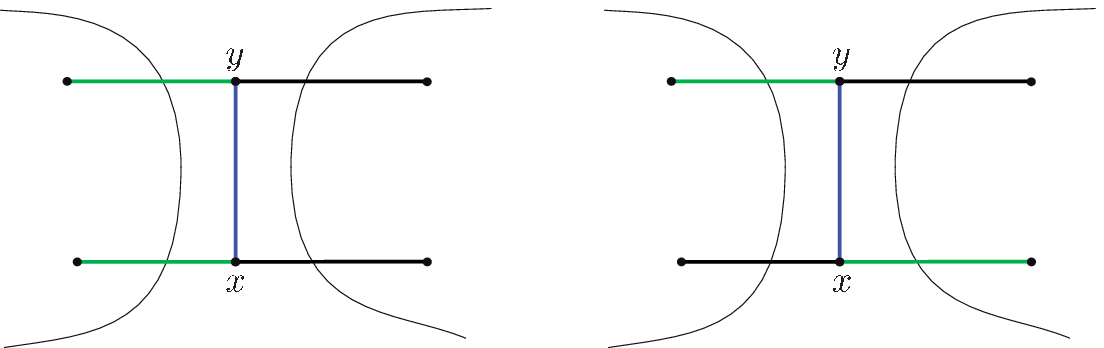}{The two cases \wrt\ the distribution of the edges incident with $xy$ into components.}

{\bf Case I:}
In the former case, we can find a \red\ presentation of \Gam\ in which no relator contains $cd$ or $cbd$ as a subword. Indeed, let $\left<b,c,d \mid R_1,R_2,\ldots\right>$ be any \red\ presentation of \Gam. Suppose that $cbd \preceq R_i$ for some $i$, let $W_i$ be a closed walk in \g induced by $R_i$, and let $e$ be a $b$ edge in the middle of a $cbd$ subpath of $W_i$. Then, as one can easily check using \fig{bcut}, $W_i$ has to visit some endvertex of $e$ twice, contradicting the fact that the presentation is \red. 

Now suppose that $cd \preceq R_i$ for some $i$, and let again $W_i$ be a closed walk in \g induced by $R_i$. Let $x$ be the middle vertex of a $cd$ subpath of $W_i$, and consider the $b$ edge $e$ incident with $x$. Since $e$ is a \sepe, we have the situation of the left part of \fig{bcut} again, and we can use $e$ to `shortcut' $W_i$ into two walks $W^1,W^2$ 
such that the concatenation of the corresponding relations $R^1,R^2$ yields a word equivalent to $R$, similarly to what we did in \fig{W}. Thus we can replace $R_i$ by two relations $R^1,R^2$ to obtain a presentation of $\Gamma$ in which the word $cd$ appears less often. Repeating this procedure as often as needed, we obtain, as claimed, a \red\ presentation $\left<b,c,d \mid R'_1,R'_2,\ldots\right>$ of \Gam\ in which $cd, cbd \not\preceq R'_i$ \fe\ $i$.

The fact that no $R'_i$ contains the words $cd$ or $cbd$ easily implies that $R'_i$ contains at most two of the letters $c,b,d$. Thus each $R'_i\neq b^2,c^2,d^2$ is of the form $(bc)^n$ or  $(bd)^n$. Since the presentation is \red, it cannot be the case that $(bc)^n,(bc)^l$ with $n\neq l$ are both relators; similarly for $(bd)^n$. Thus \Gam\ has a presentation of the form $\left<b,c,d\mid b^2, c^2,d^2, (bc)^n, (bd)^m\right>$ where $n,m$ might take the value $0$, meaning that the corresponding relator is not present. 

Let us now check that $m,n\neq 0$. Indeed, by \Lr{relcc} the  set $B$ of cycles in \g induced by the relators in the above presentation generate $\cc_f(G)$. But if $n=0$ then no element of $B$ contains a $c$ edge, which implies that no cycle of \g contains a $c$ edge. It is an easy graph-theoretical fact that every $c$ edge must be a bridge in this case, which contradicts our assumption that \g is 2-connected. Similarly, $m\neq 0$ for otherwise every $d$ edge is a bridge.

We can now apply \Cr{Macay} to obtain an embedding \sig. Indeed, the set $B$ of cycles in \g induced by the relators in the above presentation form a 2-basis, in which every $b$ edge appears twice and every other edge appears once. It is straightforward to check that every edge reverses spin in \sig\ using the fact that the elements of $B$ are precisely the finite face-boundaries.

\medskip

{\bf Case II:}
In the second case (\fig{bcut} right), using very similar arguments as above we obtain a \red\ presentation $\left<b,c,d \mid R'_1,R'_2,\ldots\right>$ of \Gam\ in which no relator contains $cbc$ or $cd$ or $dbd$ as a subword. This easily implies that the only possible kind of relator $R'_i$, except for $b^2,c^2,d^2$, is $(bcbd)^n$, and again since the presentation is \red\ only one value for $n$ is allowed. Such a relator with $n>0$ must exist for otherwise \g is not 2-connected.

Similarly to Case I, the  set $B$ of cycles in \g induced by this relator is a 2-basis of \g and applying \Cr{Macay} we obtain a \prem\ of \g in which $b$ edges preserve spin and all other edges reverse spin.
\medskip

The converse implication can, in both cases, be proved easily by a construction similar to that of \Tr{Tki}.
\end{proof}

\subsection{Graphs with no \sepe s}

We continue our analysis with the case when \g has no \sepe. The following is similar to \Lr{Lk}.

\begin{lemma}\label{Lk2}	
Let $G= Cay\left<b,c,d\mid b^2, c^2, d^2, \ldots \right>$ be a Cayley graph with $\kappa(G)=2$ and no \sepe, and let $P$ be a shortest path whose endvertices $x,y$ separate \G. Then  $P$ is two-coloured and $x^{-1}y$ is an involution.
\end{lemma}
\begin{proof}
By the same arguments as in the proof of assertion \eqref{foure} of the proof of \Lr{Lk} we can prove that
\labtequ{foure2}{The component $C$ of $G - \{x,y\}$ containing $P$ sends four edges to $\{x,y\}$.}
Let $x',y'$ be the neighbours of $x,y$ respectively in the other component $C'$ of $G - \{x,y\}$. We claim that none of the edges $xx', yy'$ has the same colour as some edge $e$ of $P$. For otherwise we can, similarly to the proof of \eqref{Pmono} in \Lr{Lk}, translate $e$ to that edge to obtain a translate of $P$ that crosses the separator $\{x,y\}$, which easily yields a contradiction to the choice of $P$. This, combined with the fact that $||P||>1$ since \g has no \sepe, implies that 
\labtequ{xxyy}{$P$ is two-coloured, with the colours $b,c$ say, and both edges  $xx', yy'$ bear the third colour $d$.}

Using this it is now easy to prove that $z=x^{-1}y$ is an involution. For if not then we obtain the situation of \fig{xyz}, where $P_1$ is a shortest \pth{x}{y}\ in $G - C$ and $P_2$ its translate by $z$. By a similar argument to that of the final part of the proof of \Lr{Lk} we can now prove that $vv'$ is also incident with $C'$, contradicting \eqref{foure2}.

\showFig{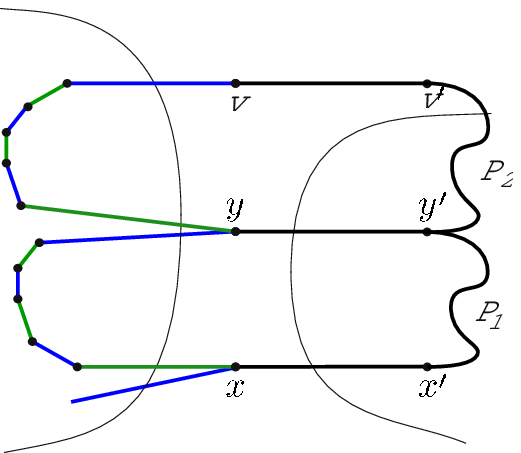}{The hypothetical situation when $x^{-1}y$ is not an involution.}
\end{proof}

\subsubsection{Graphs with no \sepe s and no 2-coloured cycles}

We will distinguish two cases according to whether \g has a cycle containing only two of the three colours $b,c,d$ or not.

\begin{lemma}\label{LkiiiA}	
Let $G= Cay\left<b,c,d\mid b^2, c^2,d^2, \ldots \right>$ be a Cayley graph with $\kappa(G)=2$ and no 2-coloured cycle and no \sepe. Then $G$ is planar and has an embedding in which all edges preserve spin.
\end{lemma}
\begin{proof}
Let $P$ be a shortest path whose endvertices $x,y$ separate \G. \Lr{Lk2} yields that $P$ is two-coloured and $z:=x^{-1}y$ is an involution. Since $z$ is an involution and \g has no 2-coloured cycle, $||P||$ must be odd. Assume \obda\ that the labels appearing in $P$ are $b$ and $c$ and that $z=b(cb)^n$. 

We are now going to construct an explicit embedding \sig\ of \G\ in the sphere. We will construct \sig\ inductively, in \oo\ steps.  
For this, let $D_0, D_1, D_2, \ldots$ be a ---possibly finite--- enumeration of the double-rays of \g spanned by $b$ and $c$. Let $G_0$ be the graph obtained from \g after contracting all edges of \g not incident to or contained in $D_0$. We claim that $G_0$ consists of the double-ray $D_0$ and an infinite set $\cp$ of pairwise disjoint paths of length two with all edges labelled $d$, each such path joining a vertex $x$ of $D_0$ to the vertex $xw$ (\fig{spiral}). Indeed, to begin with, $D_0$ is a subgraph of $G_0$. Moreover, for every vertex $x$ of $D_0$ we know that  $x,xz$ separate \G, and hence also $G_0$. Furthermore, by \eqref{xxyy} one of the components of $G - \{x,y\}$ contains $D_0$ and the other component $C'$ is incident with the two $d$ edges incident with $x$ and $xz$. Since $C'$ does not send any other edge to $D_0$, it is by the definition of $G_0$ contracted into a single vertex. This proves our claim. \fig{spiral} shows an embedding $\sig_0$ of $G_0$. For reasons that will become clear soon, consider $\sig_0$ to be an embedding in the sphere $S$ rather than in the plane. Note that all vertices of $D_0$ have the same spin in $\sig_0$. 

\epsfxsize=\hsize
\showFig{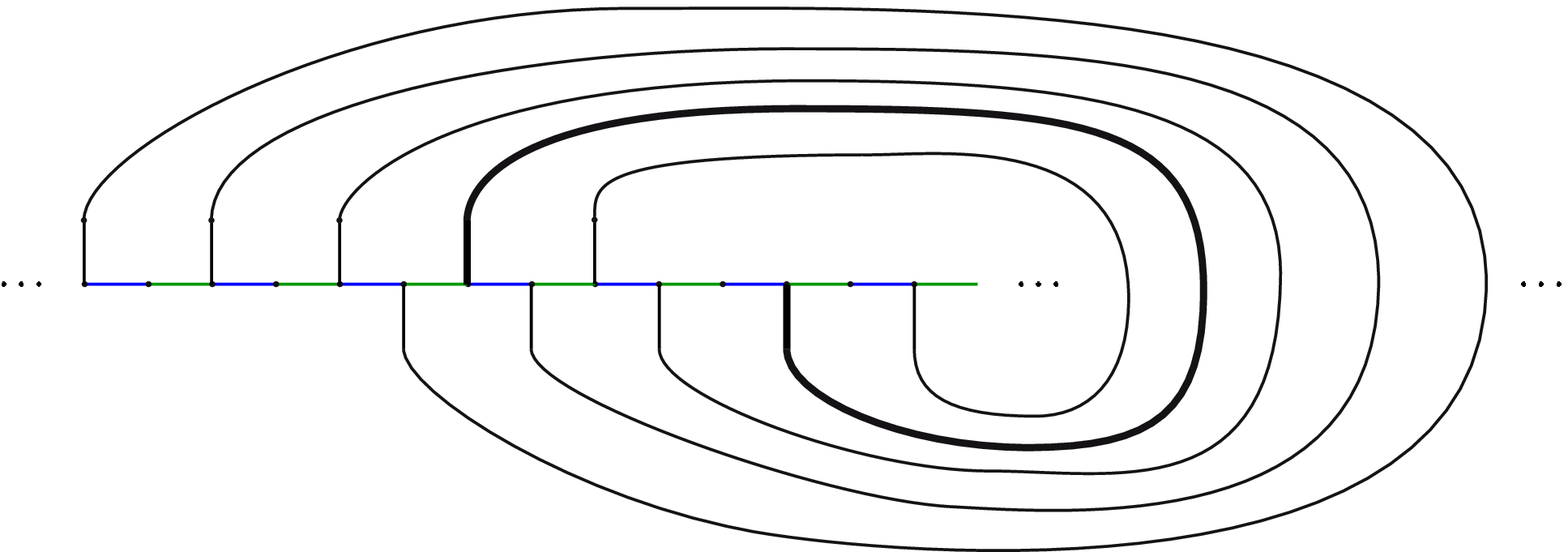}{Constructing a \prem\ in the absence of 2-coloured cycles.}

This was the first step of our inductive construction of \sig. We now proceed with the remaining steps, in each step $i$ decontracting the double-ray $D_i$ and its incident edges and extending $\sig_{i-1}$ into an embedding $\sig_i$ that includes $D_i$. 

More formally, for $i=1,2,\ldots$, let $G_i$ be the graph obtained from \g by contracting all edges that are not incident to or contained in one of the double-rays $D_0, D_1, \ldots D_i$. Consider the embedding $\sig_{i-1}$ of $G_{i-1}$ in the sphere $S$ inherited from the previous step. Let $x_i$ be the (unique) vertex of $G_{i-1}$ into which $D_i$ was contracted. Note that $x_i$ was the middle vertex of some of the   $d$-labelled paths of length two of $G_{i-1}$. Pick a closed disc $X$ in $S$ into which $\sig_{i-1}$ maps $x_i$ and part of its incident $d$-edges but no other vertex or edge. Then, consider the auxiliary graph $G'_i$, isomorphic to $G_0$, which is obtained from $G$ just like $G_0$ except that we keep $D_i$ instead of $D_0$ and contract all other double-rays. Moreover, let $\sig'_i$ be an embedding of $G'_i$ in another copy $S'$ of the sphere similar to the embedding $\sig_0$ of $G_0$.  Pick a closed disc $X'$ in $S'$ into which $\sig_{i-1}$ maps the (unique) vertex of $G'_i$ into which $D_0$ was contracted, and into which $\sig_{i-1}$ maps no other vertex. To obtain the new embedding $\sig_i$, cut both discs $X,X'$ out of their corresponding spheres $S,S'$, and glue the remainders together along their boundaries to obtain a new copy of the sphere $S''$ in which now $G_i$ is embedded (as a combination of $\sig_{i-1}$ and $\sig'_i$). Indeed, note that the boundary of each of the discs  $X,X'$ is crossed by precisely two edges, and these are the same $d$ edges $e,f$ of \g. Thus we may perform the glueing in such a way that these edges are properly embedded as arcs in $S''$, half of each edge coming from $S$ and the other half coming from $S'$. 

We would like to obtain an embedding \sig\ of \g as a limit of the embeddings $\sig_i$. To achieve this, it is more convenient to think of the embedding $\sig_i$ we just constructed as an embedding in $S$ itself rather than some new copy $S''$. Thus, formally, we define $\sig_i: \g \to S$ so that it coincides  with $\sig_{i-1}$ in $S \sm X$ and `imitates'  $\sig'_i \restr (S' \sm X')$ in $X$. We can now define the embedding \sig\ of $G$ on $S$ as the limit of the $\sig_i$: every point $p$ of \g is mapped to the unique point $s$ of $S$ \st\ $\sig_i(p)=s$ holds for all but finitely many $i$.

Next, we claim that \sig\ can be constructed so that every vertex of $G$ has the same spin.
For this, suppose that every vertex of $G_{i-1}$ has the same spin $\varsigma$ in $\sig_{i-1}$. Then every vertex of $G_{i-1}$ also has that spin in $\sig_{i}$ by construction. Moreover, all vertices of $D_{i}$ share the same spin $\varsigma'$ in $\sig_i$, as this was the case in $\sig'_i$. We may assume that $\varsigma=\varsigma'$, since if this is not the case, then we could have flipped the disc $S'\sm X'$ we replaced $X$ with around
before  performing the above cutting and glueing operation, to reverse the spins of the vertices of $D_i$ (when flipping that disc we fix the two points on its boundary that meet edges of \G). Since there are only two possible types of spin in a cubic graph, we would have then achieved $\varsigma=\varsigma'$. Thus, assuming that every vertex of $G_{i-1}$ has the same spin in $\sig_{i-1}$ we deduced that every vertex of $G_{i}$ has the same spin in $\sig_{i}$. Since this is the case at the beginning of our construction when $i=0$, we obtain by induction that every vertex of $G$ has the same spin is \sig.
\end{proof}

We can now plug the embedding \sig\ we just constructed into the following result from \cite{cayley3} to obtain a presentation of the corresponding graphs.

\begin{lemma}[\citeTIIc] \label{TIIc}
Let $G= Cay\left<b,c,d\mid b^2, c^2,d^2, \ldots \right>$ be a Cayley graph with $\kappa(G)=2$, with no 2-coloured cycle, having an embedding in which all edges preserve spin. Then $G \isom Cay \left<b,c,d\mid b^2, c^2,d^2, (b(cb)^nd)^m\right>, n\geq 1, m\geq 2$.

Conversely, each of the above presentations, with parameters chosen in the specified domains, yields a planar cubic \Cg\ of connectivity 2 with no two-coloured cycle.
\end{lemma}

\begin{corollary}\label{TkiiiA}	
Let $G= Cay\left<b,c,d\mid b^2, c^2,d^2, \ldots \right>$ be a Cayley graph with $\kappa(G)=2$ and no 2-coloured cycle and no \sepe. Then \\ $G\isom Cay \left<b,c,d\mid b^2, c^2,d^2, (b(cb)^nd)^m\right>, n,m\geq 2$. 

Conversely, each of the above presentations, with parameters chosen in the specified domains, yields a planar cubic \Cg\ of connectivity 2 with no \sepe\ and no two-coloured cycle.
\end{corollary}
\begin{proof}
 Combining \Lr{LkiiiA} with \Tr{TIIc} we obtain a presentation of the desired form, except that \Tr{TIIc} allows $n=1$ as a possibility. However, if $n=1$ then any $b$ edge would be a \sepe\ as can be easily seen with the help of \fig{spiral}: if $n=1$, in which case $|z|=3$, then no path connects the two parts of a $bc$-double-ray obtained after removing one of the $b$ edges and its endvertices. This contradicts our assumption that \g has no \sepe. On the other hand, if $n>1$ then it is easy to check that no edge is a \sepe.
\end{proof}

\comment{
	In our next result we are going to need the following result from \cite{cayley3} which, similarly to \Lr{TIIc}, provides information about the structure of \Gam\ once we know some embedding of \G.
	\note{not needed:}
	\begin{lemma}[{\cite[...]{cayley3}}] \label{TIIa2}
 Let $G= Cay\left<b,c,d\mid b^2, c^2, d^2, (bc)^n,\ldots \right>$ be a Cayley graph with $\kappa(G)=2$ and an embedding in which all edges preserve spin. Then $\Gam \isom \left< b^2, c^2, d^2 \mid bcbc, (bcd)^m\right>$.
\end{lemma}
	\begin{lemma}[{\cite[...]{cayley3}}] \label{TIId1}
 	Let $G= Cay\left<b,c,d\mid b^2, c^2, d^2, (bc)^n, \ldots \right>$ be a Cayley graph with $\kappa(G)=2$ and a \prem\ in which only $b$ preserves spin. Then $G\isom Cay \left<b,c,d\mid b^2, c^2, d^2, (bc)^n, (cbcd)^m\right>$ with $n=2k, k\geq 2$ and $m\geq 2$. Conversely, every such presentation yields a Cayley graph with $\kappa(G)=2$ and a \prem\ in which only $b$ preserves spin.
	\end{lemma}
}

\subsubsection{Graphs  having 2-coloured cycles and no \sepe s}

We can now proceed with the last result of this section, completing our characterisation of the planar cubic Cayley graphs of connectivity 2. The only remaining case is when \g has no \sepe\ but does have a two-coloured cycle. 

To begin with, let us first take the finite ones out of the way: by \cite[Chapter 27, Theorem 3.7.]{BaAut} finite simple \Cg s of degree at least 3 are \tcon, so in our case $G$ has to have parallel edges in order to have connectivity 2. It follows that if \g is finite then 
$$\Gam\isom \left<b,c,d\mid b^2, c^2, d^2, (bc)^n, cd\right>, n\geq 1.$$

From now on we can assume \g to be infinite.

\begin{theorem}\label{TkiiiB}	
Let $G= Cay\left<b,c,d\mid b^2, c^2, d^2, \ldots \right>$ be a planar Cayley graph with $\kappa(G)=2$, with a two-coloured cycle, and no \sepe. Then either $\Gam \isom \left< b,c,d| b^2, c^2, d^2, (bc)^2, (bcd)^m\right>$ with $m\geq 2$, in which case $G$ has a \prem\ with a single spin, or $\Gam\isom \left<b,c,d\mid b^2, c^2, d^2, (bc)^{2k}, (cbcd)^m\right>$, $k,m\geq 2$ and $G$ has a \prem\ in which only $c$ preserves spin. 

Conversely, each presentation as above, with parameters in the specified domains, yields a planar \Cg\ of connectivity 2 without \sepe s.
\end{theorem}
\begin{proof}
Let again $P$ be a shortest path whose endvertices $x,y$ separate \G. \Lr{Lk2} yields that $P$ is two-coloured and $z:=x^{-1}y$ is an involution. 
We claim that 
\labtequ{bcfin}{the two colours, $b,c$ say, appearing in $z$ span finite cycles.}
Indeed, if $b,c$ span double rays, then it is (at least) one of the other two pairs of colours that span two-coloured cycles. In this case, assume \obda\ that $b,d$ span cycles. We claim that $x,y$ lie on distinct $bd$ cycles. 

Suppose to the contrary that there is a $b,d$ cycle $C$ containing both $x,y$. Then $x,y$ separate $C$ into two subarcs $P_1,P_2$. Note that any \auto\ of \g mapping $x$ to some other vertex of $C$ must also map $y$ within $C$. It is now easy to see that, as none of $P_1,P_2$ can be a single edge since \g has no \sepe, there is a \auto\ $g$ of \g mapping $x$ to some vertex $x^*$ of $P_1$ and $y$ to some vertex $y^*$ of $P_2$. This, however, easily leads to a contradiction: \eqref{xxyy} implies that $\{x,y\}$ separates $P_1$ from $P_2$ in \G, and so it also separates $x^*$ from $y^*$. But there is a \pth{x^*}{y^*}\ alternating in the colours $b$ and $c$, namely the translate of $P$ under $g$, which contradicts the fact that $\{x,y\}$ only sends $d$-coloured edges to one of the components of $G - \{x,y\}$ by \eqref{xxyy}. This contradiction proves our claim that $x,y$ lie on distinct $bd$ cycles.

So let $C$ be the $bd$ cycle that contains $x$. Since $y \not \in V(C)$, all vertices of $C - x$ lie in the same component of $G - \{x,y\}$. Again, this contradicts \eqref{xxyy}, since $x$ sends both a $b$-coloured and a $d$-coloured edge to that component. Thus we have proved \eqref{bcfin}.




Assume from now on that the length of a $bc$ cycle is $2k$. We now distinguish three subcases depending on the length of $P$.

{\bf Case 1: $||P||= 2$.} In this case we have $2k=4$ since $z$ is an involution. We can now repeat the construction of \Lr{LkiiiA}, with the only difference that instead of $bc$ double rays we now have $bc$ cycles of length 4. Instead of \fig{spiral} we now have the left part of \fig{C4ii}. It is easy to check that the embedding \sig\ we obtain must look like the one in the right part of \fig{C4ii}. Using this embedding it is straightforward to check that \Gam\ can be represented by $\Gam \isom \left< b,c,d\mid b^2, c^2, d^2, bcbc, (bcd)^m\right>$ with $m\geq 2$ (for $m=1$ we obtain a degenerate case: a group with 4 elements whose \Cg\ is 3-connected).

\epsfxsize=.8\hsize
\showFig{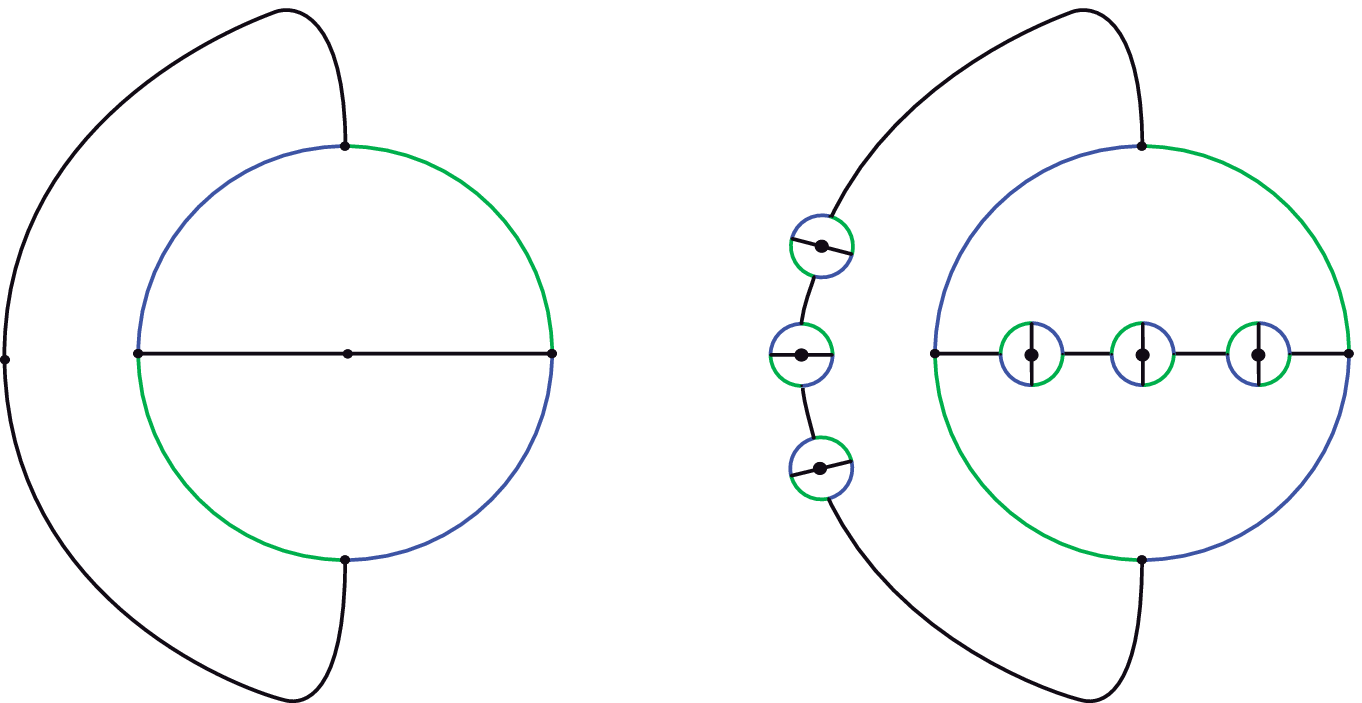}{The first step $\sig_0$ in the construction of the embedding for the graphs of Case 1 (left) and a possible later step (right). In this example we chose $m=4$, i.e.\ our graph is $Cay \left< b,c,d| b^2, c^2, d^2, bcbc, (bcd)^4\right>$.}

\medskip
{\bf Case 2: $||P||= 3$.} Easily, $2k\geq 6$ in this case. We are again going to repeat the construction of \Lr{LkiiiA}, with the $bc$ cycles playing the role of the $bc$ double-rays there, to construct a \prem\ \sig\ of \G. However, this time only one of the colours, $c$ say, will preserve spin, while all $b$ and $d$ edges will reverse spin. The first step $\sig_0$ of this construction is shown in \fig{nospiral2} (left). It is not hard to prove that if $k$ is odd then \g cannot be planar. For example, if $k=3$ then $G_0$ becomes isomorphic to the Kuratowski graph $K_{3,3}$ after suppressing the vertices of degree two.

\epsfxsize=\hsize
\showFig{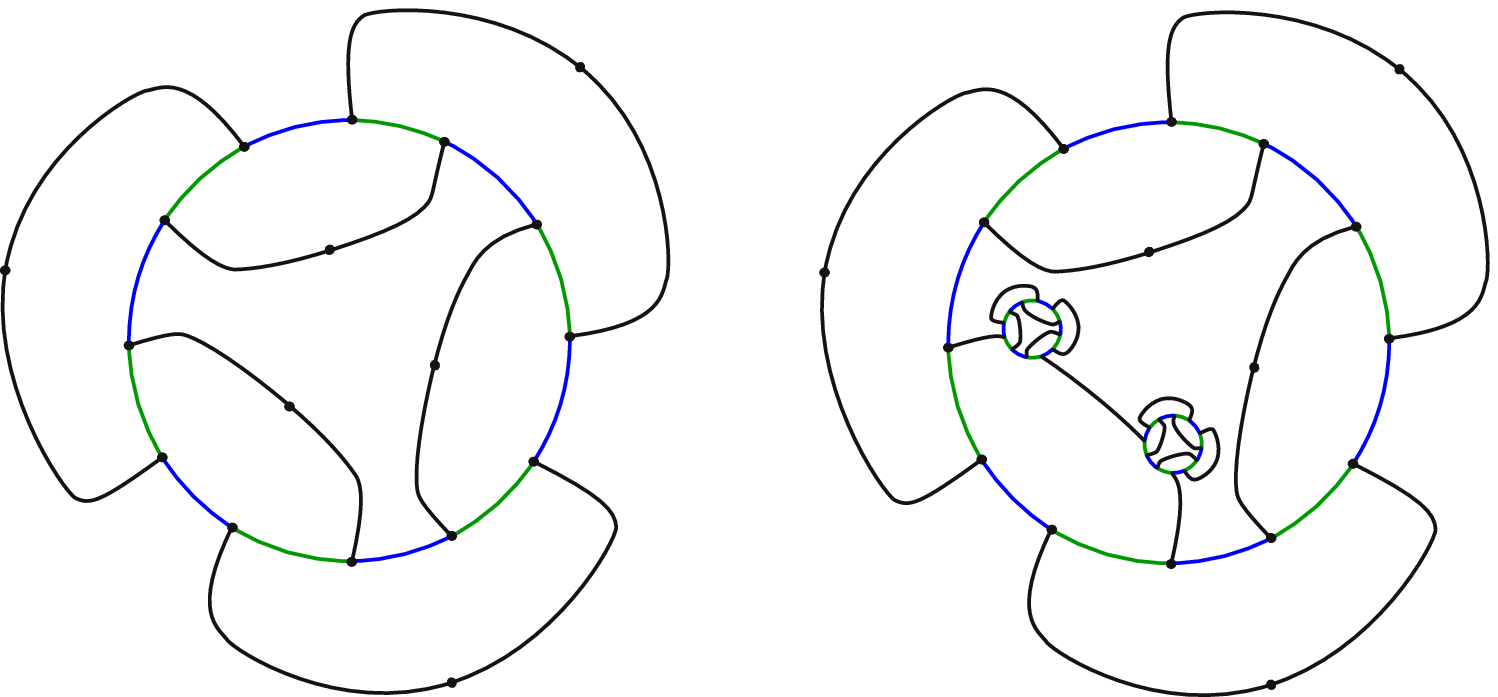}{The first step $\sig_0$ in the construction of the embedding for the graphs of Case 2 (left) and a possible later step (right). }

Note that, as desired, the $c$ edges preserve spin in this embedding, while the $b$ edges do not. Our construction  makes sure that the $b$ and $c$ edges retain this property also in the final embedding \sig. We need to make sure that all $d$ edges  reverse spin in \sig. 

Recall that in each step $i$ of our construction we decontract the cycle $D_i$, which was hidden in some vertex $x_i$ of $G_{i-1}$ that was the middle vertex of a path $P_i= y_i e x_i f z_i$ of length two, where $e$ and $f$ are labelled $d$. We will make sure by induction that the endpoints of every maximal $d$-labelled path $P$, no matter whether $P$ is a path of length two with a contracted vertex in the middle or an original edge of \G, have different spins in $\sig_i$. Note that this is the case in the embedding $\sig_0$. Now assuming that it is also true for a later embedding $\sig_{i-1}$, we can easily make sure that it remains true in $\sig_i$: for if it happens to be false, then this can only be so locally at the place where we introduced $D_i$, and we can flip the newly pasted disc $X'$ around ---as we did in the proof of \Lr{TkiiiA}--- to reverse spins to the desired state. Indeed, there are only two vertices $v,w$ (of $D_i$) in $X'$ adjacent to some vertex outside $X'$, and these two vertices have different spins because the embedding within $X'$ is by construction similar to the one of left part of \fig{nospiral2}; now as the neighbours of $v,w$ outside $X'$ have different spins by our induction hypothesis, one of the two possible ways to paste $X'$ into $\sig_{i-1}$ (i.e.\ with or without a flip) preserves our induction hypothesis that endpoints of a maximal $d$-labelled path have different spins. Thus, in the limit embedding \sig, all $d$ edges reverse spin. 

We will now use \Lr{relcc} to obtain a presentation of $\Gam$. To begin with, note that by the construction of \G, each of the faces containing a $cbc$ subpath in their boundary is induced by a relation of the form $(cbcd)^m$, where $m$ is fixed for all such faces. Let $B'$ be the set of circuits of such face boundaries, in other words, the set of cycles induced by $(cbcd)^m$. Moreover, let $B''$ be the set of circuits of the $bc$ cycles of \G. We claim that $B:= B' \cup B''$ generates $\cc_f(G)$. 

To see this, let $D$ be any cycle of $G$. Note that $D$ can be written as a sum (in $\cc_f(G)$) $D= \sum_i D_i$ of cycles $D_i$ such that no $D_i$ meets both sides of some $bc$ cycle: indeed, if $D$ meets both sides of the $bc$ cycle $C$, then we can split $D$ into two subpaths $P_1,P_2$ with endvertices on $C$, and combine each $P_i$ with a subarc $C'$ of $C$ with the same endvertices to obtain two cycles $D_1:= P_1\cup C', D_2:= P_2\cup C'$ whose sum is $D$, so that $P_1,P_2$ together cross $C$ less often than $D$ does. If some of the $D_1$ still crosses $C$, we can repeat this operation on $D_1$ and so on. As $D$ can meet at most finitely many $bc$ cycles, performing this kind of operation a finite number of times we obtain the desired presentation for $D$. It is now not hard to see that each $D_i$ can be written as a sum of elements of $B$; see \fig{nospiral2}. Thus $D= \sum_i D_i$ too can be written as a sum of elements of $B$. This proves our claim that $B$ generates $\cc_f(G)$. Applying \Lr{relcc} we now immediately obtain a presentation: $\Gam\isom \left<b,c,d\mid b^2, c^2, d^2, (bc)^{2k}, (cbcd)^m\right>$ with $k\geq 2$ and $m\geq 2$. (Again, for $m=1$ we have only one $bc$ cycle and \g is \tcon.)


Note that in the embedding we constructed, every vertex is incident with two finite and one infinite face boundary. In \cite{cayley3} we use these \Cg s as a building block in order to obtain \tcon\ plane \Cg s which have no finite face boundaries, contrary to a conjecture of Bonnington and Watkins \cite{BoWaPla} that none exist. We now state some properties of the graphs we just constructed to be used in \cite{cayley3}.

\begin{proposition}
Let $\g = Cay \left<b,c,d\mid b^2, c^2, d^2, (bc)^{2k}, (cbcd)^m\right>$. Then the following assertions are true.
\begin{enumerate} 
\item \label{nosi} \G\ has a \prem\ $\sig$ in which $c$ preserves spin while $b, d$ reverse spin. In this embedding, each vertex is incident with two faces bounded by a cycle induced by the relator $(cbcd)^m$ and one face that has infinite boundary.
\item \label{nosii} \G\ cannot be separated by removing two edges $e,f$ unless both $e,f$ are coloured $d$, and it cannot be separated by removing a vertex and an edge $e$ unless $e$ is coloured $d$; 
\item \label{nosiii} If a pair of vertices $s,t$ of a cycle $C$ of \G\ induced by 
$(cbcd)^m$ separates \G, then both $s,t$ are incident with a $d$-edge of $C$;
\item \label{nosiv} \G\ has no \sepe;
\item \label{nosvi}	\fe\  cycle $C$ of \g induced by the word $(cbcd)^m$, and every $b$ edge $vw$ of $C$, \ti\  a \pth{v}{w}\ in $G$ meeting $C$ only at $v,w$, and
\item \label{nosv} If two cycles $C,D$ of \G\ induced by the word $(cbcd)^m$ share an edge $uv$, then \ti\ path from $C$ to $D$ in $\G - \{u,v\}$.
\end{enumerate}
\end{proposition}
\begin{proof}
The \prem\ $\sig$ required by \ref{nosi} was constructed above. 

To prove that \ref{nosii} is true, it suffices to show that if $uv$ is an edge coloured $b$ or $c$ then there are two independent \pths{u}{v}\ none of which is $uv$ itself. This is indeed true, and is easy to check using \fig{nospiral2}.

To see that \ref{nosiii} is true, suppose that $s$ is not incident with  a $d$-edge of $C$, in which case it must be incident with  a $b$-edge of $C$. Consider the two neighbours $u=sc$ anf $v=sb$ of $s$ on $C$, and note that if $\{s,t\}$ separates \G, then $u,v$ lie in distinct component because of the cycle $C$. However, it is easy to find a \pth{u}{v}\ in \g that does not meet $C$ except at $u$ and $v$ (use the $bc$ cycle containing $uv$). This easily yields a contradiction.

It is easy to check \ref{nosiv} and \ref{nosvi} using \fig{nospiral2}.

Finally, note that if two cycles $C,D$ of \G\ induced by the word $(cbcd)^m$ share an edge $uv$, then the colour of  $uv$ is $c$, and the $bc$ cycle containing $uv$ has a subpath joining $C$ to $D$ without using any of $u,v$ as claimed by \ref{nosv}.
\end{proof}

{\bf Case 3: $||P||> 3$.} Easily, $2k\geq 8$ in this case. It is now easy to check that \g cannot be planar: using one of the $b,c$ cycles $C$ and three translates of $P$ with endpoints on $C$ one can force $K_{3,3}$, one of the Kuratowski graphs, as a topological minor of \G.

\medskip
To prove the converse implication, that every presentation as in \Tr{TkiiiB} yields a planar Cayley graph without \sepe s, we can use the constructions of \figs{C4ii} and \ref{nospiral2} to explicitly construct planar \Cg s with embeddings as in the assertion. The fact that these Cayley graphs do indeed have the desired presentation follows from the forward implication, which we have already proved.
\end{proof}

We point out the following observation, which is a consequence of our analy\-sis, to be used in \cite{cayley3}.
\begin{observation}
No graph of the types \ref{iv} or \ref{v} of \Tr{Tmain} has an embedding in which some 2-coloured cycle bounds a face.  
\end{observation}

\section{Summary and final remarks}

Combining our analysis of Sections \ref{sec2} and \ref{sec3} immediately yields our main results:

\begin{proof}[Proof of \Tr{Tmain}]
Any cubic \Cg\ \g has either two or three generators. In the former case Theorems \ref{Tki} and \ref{Tkiv} imply that \g is of type \ref{i}, \ref{ii} or \ref{iii}. In the latter case, Theorems \ref{Tkii} \ref{TkiiiA} and \ref{TkiiiA} imply that \g is of one of the types \ref{iv} to \ref{ix}.

We need to check that \g cannot belong to more than one of these types. For most of the pairs of types this is obvious; we now consider the remaining pairs. The fact that \g cannot be of both type \ref{i} and \ref{ii} follows from assertion \ref{noaba} of the proof of \Tr{Tki}. The graphs of the last two types have no 2-coloured cycles, thus they cannot also belong to some of the types  \ref{iv} to \ref{vi}. \Tr{TkiiiB} shows that no graph of both types \ref{iv} and \ref{v} can exist.  

The converse implication is again obtained by combining the corresponding results in Sections \ref{sec2} and \ref{sec3}.
\end{proof}

Similarly, \Cr{cornp} follows easily from propositions \ref{Tki}, \ref{Tkii}, \ref{Lk2} and \ref{LkiiiA}.
\medskip

Our analysis of Sections \ref{sec2} and \ref{sec3} also implies 
\begin{corollary} \label{cprem}
Every cubic planar \Cg\ of connectivity 2 admits a \prem.
\end{corollary}

An important property of planar a Cayley graph  is whether it admits a \vapf\ embedding (as defined in \Sr{secDem}), since this can have important implications for its group-theoretical \cite{vapf} as well as its graph-theoretical properties \cite{CWY, kozPPP}. For the graphs studied in this paper we can always decide whether they have this property:

\begin{corollary} \label{cvap}
A cubic planar \Cg\ of connectivity 2 admits a \vapf\ embedding \iff\ it does not belong to one of the types \ref{iii}, \ref{iv}, \ref{v}, and \ref{vii} of \Tr{Tmain}.
\end{corollary}
The interested reader will be able to prove \Cr{cvap} using our analysis of these graphs. 

It is proved in \cite{vapf} that a \Cg\ admits a \vapf\ embedding \iff\ it is the 1-skeleton of a Cayley complex that can be embedded in the plane after removing some redundant simplices. Thus, by the above corollary, some of our graphs are not 1-skeletons of any such Cayley complex. However, they are 1-skeletons of an \defi{almost planar} Cayley complex, that is, a Cayley complex that can be mapped to $\R^2$ in such a way that the images of the interiors of any two 2-simplices are either disjoint or one of them is contained in the other, or their intersection is a 2-simplex bounded by the two parallel edges corresponding to some involution in the generating set. 

\begin{corollary} \label{cplem}
Every cubic planar \Cg\ of connectivity 2 is the 1-skeleton of an almost planar Cayley complex.
\end{corollary}
This fact can be easily seen by considering the embeddings we constructed and the presentations we chose.

\Cr{cplem} is extended  in \cite{cayley3} to all cubic planar \Cg s by more involved arguments.

\acknowledgement{I am grateful to Martin Dunwoody for useful discussions on the topic.}

\bibliographystyle{plain}
\bibliography{../collective}

\begin{thebibliography}{10}

\bibitem{BaAut}
L.~Babai.
\newblock {Automorphism groups, isomorphism, reconstruction.}
\newblock In {\em {Graham, R. L. (ed.) et al., Handbook of combinatorics. Vol.
  2. Amsterdam: Elsevier (North-Holland)}}, pages 1447--1540. 1995.

\bibitem{bogop}
O.~Bogopolski.
\newblock {\em Introduction to Group Theory}.
\newblock EMS, Zuerich, Switzerland, 2008.

\bibitem{BoWaPla}
C.~P. Bonnington and M.~E. Watkins.
\newblock Planar embeddings with infinite faces.
\newblock {\em Journal of Graph Theory}, 42(4):257--275, 2003.

\bibitem{CWY}
Q.~Cui, J.~Wang, and X.~Yu.
\newblock Hamilton circles in infinite planar graphs.
\newblock {\em J.~Combin.\ Theory (Series B)}, 99(1):110--138, 2009.

\bibitem{diestelBook05}
R.~Diestel.
\newblock {\em Graph {T}heory \emph{(3rd edition)}}.
\newblock Springer-Verlag, 2005.
\newblock \\ Electronic edition available at:\\ {\small\tt
  http://www.math.uni-hamburg.de/home/diestel/books/graph.theory}.

\bibitem{DrSeSeCon}
C.~Droms, B.~Servatius, and H.~Servatius.
\newblock {Connectivity and planarity of Cayley graphs.}
\newblock {\em Beitr.\ Algebra Geom.}, 39(2):269--282, 1998.

\bibitem{dunPla}
M.J. Dunwoody.
\newblock Planar graphs and covers.
\newblock Preprint.

\bibitem{vapf}
A.~Georgakopoulos.
\newblock A group has a flat cayley complex if and only if it has a {VAP}-free
  cayley graph.
\newblock Preprint 2010.

\bibitem{cayley3}
A.~Georgakopoulos.
\newblock The planar cubic cayley graphs.
\newblock Preprint 2011.

\bibitem{am}
A.~Georgakopoulos.
\newblock Word extensions of groups.
\newblock In preparation.

\bibitem{agmh}
A.~Georgakopoulos and M.~Hamann.
\newblock In preparation.

\bibitem{kozPPP}
G.~Kozma.
\newblock Percolation, perimetry, planarity.
\newblock {\em Rev. Mat. Iberoamericana}, 23(2):671--676, 2007.

\bibitem{whitney_congruent_1932}
H.~Whitney.
\newblock Congruent graphs and the connectivity of graphs.
\newblock {\em American J.\ of Mathematics}, 54(1):150--168, January 1932.

\end{thebibliography}
\end{document}